\documentclass[reqno,12pt]{amsart}
\usepackage{amssymb,amsthm,amsmath,amsfonts,amstext,mathrsfs,
url,fancybox}
\usepackage{latexsym}
\usepackage{fancyhdr}
\usepackage{epsfig}
\usepackage[unicode]{hyperref}

\usepackage{pgf}

\usepackage{verbatim}

\hypersetup{colorlinks=true,
linkcolor=blue,
citecolor=blue,
urlcolor=blue,
filecolor=blue,
bookmarksnumbered=true,
pdfstartview=FitH,
pdfhighlight=/N
}

\newtheorem{cor}{Corollary}

\newtheorem{prop}{Proposition}
\newtheorem*{defin}{Definition}

\newcounter{noteno}
\newcounter{exam}

\newenvironment{Note}%
	{\refstepcounter{noteno}%
	\begin{small}
	\medbreak\par\noindent{{\bf Note~\thenoteno}.}}%
	{\hfill{$\Box$}\end{small}\par\medbreak}

	{\refstepcounter{exam}%
	\begin{small}
	\medbreak\par\noindent{{\bf Example~\theexam}.}}%
	{\hfill{$\Box$}\end{small}\par\medbreak}

\input xy
\xyoption{all}

\newcommand{\m}{\mathbf}
\newcommand{\bl}{\bullet}
\newcommand{\p}{\partial}
\def\d{\,{\rm{d}}}

\title[Projective superflows. III]
{Projective superflows. III.\\
Finite subgroups of $U(2)$}
\author[G. Alkauskas]{Giedrius Alkauskas}
\address{Vilnius University, Department of Mathematics and Informatics, Naugarduko 24, LT-03225 Vilnius, Lithuania}
\email{giedrius.alkauskas@mif.vu.lt}

\begin{document}
\begin{abstract} Let $\m{x}\in\mathbb{R}^{n}$ or $\mathbb{C}^{n}$. For $\phi:\mathbb{R}^{n}\mapsto\mathbb{R}^{n}$ (respectively, $\phi:\mathbb{C}^{n}\mapsto\mathbb{C}^{n}$) and $t\in\mathbb{R}$ (respectively, $\mathbb{C}$), we put $\phi^{t}=t^{-1}\phi(\m{x}t)$. \emph{A projective  flow} is a solution to the projective translation equation $\phi^{t+s}=\phi^{t}\circ\phi^{s}$, $t,s\in\mathbb{R}$ or $\mathbb{C}$.\\ \indent \emph{The projective superflow} is a projective flow with a rational vector field which, among projective flows with a given symmetry, is, up to a homothety, unique and optimal.\\
\indent In the first and the second part of this work, we classified real $2$ and $3-$dimensional supeflows over $\mathbb{R}$. \\
\indent In this third part we classify all $2-$dimensional complex superflows; that is, whose group of symmetries are finite subgroups of $U(2)$. This includes both irreducible and reducible superflows. 
\end{abstract}
\pagestyle{fancy}
\fancyhead{}
\fancyhead[LE]{{\sc Projective superflows. III.}}
\fancyhead[RO]{{\sc G. Alkauskas}}
\fancyhead[CE,CO]{\thepage}
\fancyfoot{}

\date{August 8, 2016}
\thanks{The research of the author was supported by the Research Council of Lithuania grant No. MIP-072/2015}

\maketitle
\section{Projective flows}
\subsection{Preliminaries}
\label{prelim}
Let $\m{x}\in\mathbb{C}^{n}$, $\phi:\mathbb{C}^{n}\mapsto\mathbb{C}^{n}$ (or, respectively, everywhere $\mathbb{R}^{n}$). The introduction to the problem and general setting is contained in the first part of this study. We only recall that \emph{the projective translation equation} was first introduced in \cite{alkauskas-t} and is the equation of the form

\begin{eqnarray}\setlength{\shadowsize}{2pt}\shadowbox{$\displaystyle{\quad
\frac{1}{t+s}\,\phi\big{(}\m{x}(t+s)\big{)}=\frac{1}{s}\,\phi\Big{(}\phi(\m{x}t)\frac{s}{t}\Big{)}},\quad t,s\in\mathbb{R}\text{ or }\mathbb{C}$.\quad}\label{funk}
\end{eqnarray}

A non-singular solution of this equation is called \emph{a projective flow}.  The \emph{non-singularity} means that a flow satisfies the boundary condition
\begin{eqnarray}
\lim\limits_{t\rightarrow 0}\frac{\phi(\m{x}t)}{t}=\m{x}.
\label{init}
\end{eqnarray}

A smooth $2-$dimensional projective flow $\phi=u\bl v$ is accompanied by its \emph{vector field}, which is found from
\begin{eqnarray}
\varpi(x,y)\bl\varrho(x,y)=\frac{\d}{\d t}\frac{\phi(xt,yt)}{t}\Big{|}_{t=0}.
\label{vec}
\end{eqnarray}
We use this notation instead of the standard $\varpi\frac{\p}{\p x}+\varrho\frac{\p}{\p y}$, since in our setting we consider this as a vector, not as a derivation. Vector field is necessarily a pair of $2$-homogenic functions. \\

Apart from the first paper \cite{alkauskas-t} where topologic side of the projective translation equation is considered, in \cite{alkauskas, alkauskas-un, alkauskas-ab} we only investigate vector fields which are rational functions, necessarily $2-$homogenic. So, this field of research is in the intersection of differential geometry, algebraic geometry, theory of algebraic and abelian functions, special functions of mathematical physics.\\

If a function is smooth, the functional equation (\ref{funk}) implies the PDE \cite{alkauskas}
\begin{eqnarray}
u_{x}(\varpi-x)+u_{y}(\varrho-y)=-u,\label{pde}
\end{eqnarray}
and the same PDE for $v$, with the boundary conditions as given by (\ref{init}). Namely,
\begin{eqnarray*}
\lim\limits_{t\rightarrow 0}\frac{u(xt,yt)}{t}=x,\quad 
\lim\limits_{t\rightarrow 0}\frac{v(xt,yt)}{t}=y.
\end{eqnarray*}
The orbits of the flow with the vector field $\varpi\bl\varrho$ are given by $\mathscr{W}(x,y)=\mathrm{const}.$, where the function $\mathscr{W}$ can be found from the differential equation
\begin{eqnarray}
\mathscr{W}(x,y)\varrho(x,y)+\mathscr{W}_{x}(x,y)[y\varpi(x,y)-x\varrho(x,y)]=0.
\label{orbits}
\end{eqnarray}
The function $\mathscr{W}$ is uniquely determined (up to a scalar multiple) from this ODE and the condition that it is a $1-$homogenic function. If there exists $N\in\mathbb{N}$ such that $\mathscr{W}^{N}(x,y)$ is a rational function, necessarily $N-$homogenic in $x,y$, such minimal $N$ is call \emph{a level of the flow} \cite{alkauskas}. The orbits of the flow are then algebraic curves $\mathscr{W}^{N}(x,y)=\mathrm{const}.$, and the flow itself can be integrated in terms of abelian functions. Generically, this holds for $2-$dimensional flows with rational vector fields with rational coefficients \cite{alkauskas-ab}. However, the orbits of the flow with the vector field $x^2+xy+y^2\bl xy+y^2$ are non-algebraic curves $\exp\big{(}-\frac{x}{y}-\frac{x^2}{2y^2}\big{)}y=\mathrm{const.}$  \cite{alkauskas-ab}.

\subsection{Superflows}
We briefly recall the notions of the superflows introduced in \cite{alkauskas-un,alkauskas-super1} and continued in \cite{alkauskas-super2}.
\begin{defin}
Let $n\in\mathbb{N}$, $n\geq 2$, and $\Gamma\hookrightarrow{\rm GL}(n,\mathbb{R})$ or ${\rm GL}(n,\mathbb{C})$, be an exact representation of a finite group, and we identify $\Gamma$ with the image. We call the flow $\phi(\m{x})$ \emph{the $\Gamma$-superflow}, if 
\begin{itemize}
\item[i)]there exists a vector field $\mathbf{Q}(\m{x})=Q_{1}\bl\cdots\bl Q_{n}\neq 0\bl\cdots\bl 0$ whose components are $2$-homogenic rational functions and which is exactly the vector field of the flow $\phi(\m{x})$, such that
\begin{eqnarray}
\gamma^{-1}\circ\mathbf{Q}\circ\gamma(\mathbf{x})=
\mathbf{Q}(\mathbf{x})\label{kappa}
\end{eqnarray}
 is satisfied for all $\gamma\in\Gamma$,
\item[ii)] every other vector field $\mathbf{Q}'$ which satisfies (\ref{kappa}) for all $\gamma\in\Gamma$ is either a scalar multiple of $\mathbf{Q}$, or  its degree of a common denominator is higher than that of $\mathbf{Q}$. 
\end{itemize}

The superflow is said to be \emph{reducible or irreducible}, if the representation  $\Gamma\hookrightarrow{\rm GL}(n,\mathbb{C})$ ($\Gamma\hookrightarrow{\rm GL}(n,\mathbb{R})$ is also considered as a complex representation) is reducible or, respectively, irreducible.
\end{defin} 
Thus, if $\phi$ is a superflow,  then it is uniquely defined up to conjugation with a linear map $\m{x}\mapsto t\m{x}$. This corresponds to multiplying all components of $\mathbf{Q}$ by $t$.\\

In a $2-$dimensional real case, we found that for every $d\in\mathbb{N}$, there exist the superflow whose group of symmetries is the dihedral groups $\mathbb{D}_{4d+2}$. This list exhaust all $2-$dimensional superflows \cite{alkauskas-super1}. \\

In a $3-$dimensional real case, we found that there exists three superflows \cite{alkauskas-super1,alkauskas-super2}:
\begin{itemize}
\item[$\widehat{\mathbb{T}}$)] The superflow whose group of symmetries is of order $24$, the group of full symmetries of a tetrahedron, generic orbits are space curves of arithmetic genus $1$, and the superflow itself can be described in terms of Jacobi elliptic functions;
\item[$\mathbb{O}$)]The superflow whose group of symmetries is of order $24$, the group of orientation preserving symmetries of an octahedron,  generic orbits are space curves of arithmetic genus $9$, and the superflow itself can be described in terms of Weierstrass elliptic functions;
\item[$\mathbb{I}$)]The superflow whose group of symmetries is of order $60$, the group of orientation preserving symmetries of an icosahedron,  generic orbits are space curves of arithmetic genus $25$.
\end{itemize}
All these superflows are irreducible.
\section{Reducible superflows. I. A special case}
\subsection{$2-$dimensional reducible representations}If an exact representation $\Gamma\hookrightarrow{\rm GL}(2,\mathbb{C})$ is reducible, then the group $\Gamma$ is cyclic, and it is conjugate to the group generated by
\begin{eqnarray*}
\alpha=\begin{pmatrix}
\zeta & 0\\
0 &\xi
\end{pmatrix},\quad\zeta^{m}=\xi^{n}=1\quad m,n\in\mathbb{N},
\quad |\Gamma|=\frac{mn}{\mathrm{g.c.d.}(m,n)}.
\end{eqnarray*}
Here we assume that $\zeta$ and $\xi$ are primitive $m$th and $n$th roots of unity, respectively.\\   

As an introduction to this setting, in this section we will investigate a special case which produces the first non-trivial examples of reducible superflows. As noted in \cite{alkauskas-super1}, the case $\xi=\zeta^{-1}$ never gives a superflow. Indeed, such an invariant vector field is necessarily of the form $\varpi(x,y)\bl\varpi(y,x)$, and thus the vector field $\varpi(x,y)\bl c\varpi(y,x)$, $c\neq 0$, is also invariant under conjugation with $\alpha$, and this contradicts the uniqueness property of the superflow. \\
    
Let $m\in\mathbb{N}$, $\zeta=e^{\frac{2\pi i}{m}}$ be a primitive $m$th root of unity, and let
\begin{eqnarray}
\alpha=\begin{pmatrix}
\zeta & 0\\
0 & -\zeta^{-1}
\end{pmatrix}.\label{alfa}
\end{eqnarray}
We have: $n=m$, $n=\frac{m}{2}$ or $n=2m$, depending on whether $m\equiv0\text{ (mod 4})$, $m\equiv2\text{ (mod 4})$, or $m$ is odd. The trace of the matrix $\alpha$ is equal to $2i\sin\frac{2\pi}{m}$, for $m\geq 3$ this is not a real number, and so for $m\geq 3$ the matrix $\alpha$ is not conjugate to any matrix in $\mathrm{GL}(2,\mathbb{R})$.
  
\subsection{The first special case}
\label{spec-first}
First we will investigate the case $m=4k+3$, $k\in\mathbb{N}\cup\{0\}$.
\begin{prop}The only vector field with a denominator of degree $\leq 2k$ which is  invariant under conjugation with $\alpha$ is given by, up to multiplication with a constant, by
\begin{eqnarray*}
\varpi\bl\varrho=\frac{y^{2k+2}}{x^{2k}}\bl 0.
\end{eqnarray*}
So this vector field produces the superflow. 
\end{prop}
\begin{proof}The ring of invariants for the group $\Gamma=\{\alpha^{s}:s=0,\ldots,8k+5\}$ is generated by $x^{4k+3}$, $y^{8k+6}$, $(xy)^2$. We know that the denominator of the superflow is a relative invariant of the group. The group $\Gamma$ has the following relative invariants of degree up to $2k$: $x^{i}y^{j}$, $i,j\geq 0$, $i+j\leq 2k$. Indeed, the smallest degree relative invariant (in fact, an invariant) which is not a monomial, is equal to $x^{8k+6}+cy^{8k+6}$. Without loss of generality, we can investigate only cases when a denominator is equal to $x^{\ell}y^{2k-\ell}$, $\ell=0,1,\ldots,2k$. \\

Let therefore
\begin{eqnarray*}
\widehat{\varpi}\bl\widehat{\varrho}=\frac{P(x,y)}{x^{\ell}y^{2k-\ell}}\bl\frac{Q(x,y)}{x^{\ell}y^{2k-\ell}}
\end{eqnarray*} 
be a generic $2-$homogenic vector field with a denominator $x^{\ell}y^{2k-\ell}$, where
\begin{eqnarray}
P(x,y)&=&\sum\limits_{i=0}^{2k+2}u_{i}x^{i}y^{2k+2-i},\quad u_{i}\in\mathbb{C},\label{pol-p}\\
Q(x,y)&=&\sum\limits_{i=0}^{2k+2}v_{i}x^{i}y^{2k+2-i},\quad v_{i}\in\mathbb{C}\label{pol-q}.  
\end{eqnarray}  
We know that if there exist a superflow with a denominator $x^{\ell}y^{2k-\ell}$, its vector field is given by \cite{alkauskas-super1}
\begin{eqnarray*}
\varpi\bl\varrho=\frac{1}{|\Gamma|}\sum\limits_{\sigma\in\Gamma}\sigma^{-1}
\circ(\widehat{\varpi}\bl\widehat{\varrho})\circ{\sigma}(x,y).
\end{eqnarray*}
So, the vector field of a superflow, up to a scalar multiple $\frac{1}{|\Gamma|}=\frac{1}{8k+6}$ (which can be discarded) is given by
\begin{eqnarray}
\sum\limits_{s=0}^{8k+5}\frac{\zeta^{-s}P\big{(}\zeta^{s}x,(-1)^s\zeta^{-s}y\big{)}}{\big{(}\zeta^{s}x\big{)}^{\ell}\Big{(}(-1)^{s}\zeta^{-s}y\Big{)}^{2k-\ell}}\bl \sum\limits_{s=0}^{8k+5}\frac{(-1)^{s}\zeta^{s}Q\big{(}\zeta^{s}x,(-1)^s\zeta^{-s}y\big{)}}{\big{(}\zeta^{s}x\big{)}^{\ell}\Big{(}(-1)^{s}\zeta^{-s}y\Big{)}^{2k-\ell}}.
\label{sum}
\end{eqnarray}
Consider now the term $x^{i}y^{2k+2-i}$ appearing in the polynomial (\ref{pol-p}). The part of the sum (the coefficient at $x^{i-\ell}y^{2+\ell-i}$) which corresponds to it in (\ref{sum}) is equal to (recall that $\zeta^{4k+3}=1$)
\begin{eqnarray}
\sum\limits_{s=0}^{8k+5}\zeta^{2is-2\ell s-3s}(-1)^{s(2+\ell-i)}
=
\frac{\Big{(}(-1)^{i+\ell}\zeta^{2i-2\ell-3}\Big{)}^{8k+6}-1}{(-1)^{i+\ell}\zeta^{2i-2\ell-3}-1}=0,
\label{sum-pirm}
\end{eqnarray}
unless $(-1)^{i+\ell}\zeta^{2i-2\ell-3}=1$; then it is non-zero. The latter happens when $i+\ell$ is even ($\zeta^{t}$ is never equal to $-1$), and $2i-2\ell-3\equiv 0\text{ (mod }4k+3)$. Since $0\leq i\leq 2k+2$, $0\leq\ell\leq 2k$, the only pair to satisfy these is $(i,\ell)=(0,2k)$. \\

Equally, consider now the term $x^{i}y^{2k+2-i}$ appearing in the polynomial (\ref{pol-q}). The part of the sum which corresponds to it in (\ref{sum}) is equal to
\begin{eqnarray}
\sum\limits_{s=0}^{8k+5}(-1)^s\zeta^{2is-2\ell s-s}(-1)^{s(2+\ell-i)}
=
\frac{\Big{(}(-1)^{i+\ell+1}\zeta^{2i-2\ell-1}\Big{)}^{8k+6}-1}{(-1)^{i+\ell+1}\zeta^{2i-2\ell-1}-1}=0,
\label{sum-antr}
\end{eqnarray}
unless $(-1)^{i+\ell+1}\zeta^{2i-2\ell-1}=1$; then it is non-zero. No pair $(i,\ell)$ with $0\leq i\leq 2k+2$, $0\leq\ell\leq 2k$ satisfies this. \emph{A posteriori}, the crucial thing here is when $i=2k+2$, $\ell=0$, $\zeta^{2i-2\ell-1}=1$, but $(-1)^{i+\ell+1}=-1$. That is why $\xi=\zeta$ produces no superflow, while $\xi=-\zeta^{-1}$ does.\\

Thus, for any vector field, the sum (\ref{sum}) is non-zero and has one free coefficient only when $(i,\ell)=(0,2k)$, and thus $\frac{y^{2k+2}}{x^{2k}}\bl 0$ is the unique non-zero (up to multiplication by a constant) vector field invariant under conjugation with $\alpha$ whose denominator is of the form $x^{\ell}y^{2k-\ell}$. Thus, $\frac{y^{2k+2}}{x^{2k}}\bl 0$ satisfies all the conditions imposed upon the vector field of the superflow.  
\end{proof}
We can integrate the vector field $\varpi\bl\varrho=\frac{y^{2k+2}}{x^{2k}}\bl 0$ immediately, based on the method developed in \cite{alkauskas}. Let $U(x,y)\bl y$ be this flow. Note that the second coordinate is equal to $0$ due to $\varrho=0$. We know that 
\begin{eqnarray*}
\int\limits_{\frac{y}{U(x,y)}}^{\frac{y}{x}}\frac{\d t}{\varpi(1,t)}=y,
\end{eqnarray*}
where $\varpi(1,t)=t^{2k+2}$. We can multiply the vector field by a constant without changing the superflow property. Thus, for the sake of simplicity, we may assume $\varpi(1,t)=\frac{1}{2k+1}t^{2k+2}$. This gives
\begin{eqnarray*}
\Bigg{(}\frac{U(x,y)}{y}\Bigg{)}^{2k+1}-\Big{(}\frac{x}{y}\Big{)}^{2k+1}=y.
\end{eqnarray*}
Thus, we have proved the first part of the following.
\begin{prop}
\label{prop2}
Let $k\in\mathbb{N}$. Let us define
\begin{eqnarray*}
\phi(\m{x})=U(x,y)\bl V(x,y)=\sqrt[2k+1]{x^{2k+1}+y^{2k+2}}\bl y.
\end{eqnarray*}  
Then $\phi$ has a vector field $\frac{1}{2k+1}\frac{y^{2k+2}}{x^{2k}}\bl 0$, and is a projective superflow for the cyclic group of order $8k+6$ generated by the matrix $\alpha$. The full group of symmetries of $\phi$ is given by
\begin{eqnarray*}
\Gamma_{4k+3}=\Bigg{\{}\gamma_{c}=\begin{pmatrix}
c^{2k+2} & 0\\
0 & c^{2k+1}
\end{pmatrix}:\quad c\in\mathbb{C}^{*}\Bigg{\}}.
\end{eqnarray*}
\end{prop}
This is the second example of reducible superflows. No superflows in \cite{alkauskas-super1, alkauskas-super2} are reducible, apart from the superflow
\begin{eqnarray*}
\frac{x}{x+y+1}\bl\frac{y}{x+y+1},
\end{eqnarray*}
described in (\cite{alkauskas-super1}, Subsection 2.2). The latter is a rational flow of level $0$ (its orbits are curves $\frac{x}{y}=\mathrm{const.}$), while $\phi(\m{x})$ in Proposition \ref{prop2} is a flow of level $1$. Indeed, its orbits are given by $y=\mathrm{const}$.\\

For $\phi(\m{x})$, we can check the flow property (\ref{funk}) immediately.

\begin{Note} We see that irreducible superflows have always a finite group as the group of their symmetries. For reducible flows, as far as our examples suggest, the following phenomenon occurs: the superflow itself is defined in terms of a finite group, but the whole group of symmetries turns out to be infinite, an extension of an initial one. This dichotomy \emph{finite symmetry group - infinite symmetry group} seems to be essential in distinguishing \emph{irreducible - reducible} superflows.   
\end{Note}
\begin{proof} We are only left to find all symmetries. Let $k\geq 1$. Suppose that $L$ is a non-degenerate linear map, and
\begin{eqnarray*}
L^{-1}\circ(\frac{y^{2k+2}}{x^{2k}}\bl 0)\circ L(x,y)=\frac{y^{2k+2}}{x^{2k}}\bl 0.
\end{eqnarray*} 
We directly see that $L$ is a diagonal linear transformation, and it is easy to see that it is of the form given in the formulation of the Proposition. This vector field has no other linear symmetries. We will soon see that this is not true in case $k=0$. 
\end{proof}  
We can formally verify the invariance of $\phi$ under conjugation with $\gamma_{c}$. Indeed,
\begin{eqnarray*}
\gamma_{c}^{-1}\circ(U\bl V)\circ\gamma_{c}(\m{x})&=&
\gamma_{c}^{-1}\circ\Big{(}U(c^{2k+2}x,c^{2k+1}y)\bl V(c^{2k+2}x,c^{2k+1}y)\Big{)}\\
&=&\gamma_{c}^{-1}\Big{(}\sqrt[2k+1]{x^{2k+1}c^{(2k+1)(2k+2)}+y^{2k+2}c^{(2k+1)(2k+2)}}\bl c^{2k+1}y\Big{)}\\
&=&\sqrt[2k+1]{x^{2k+1}+y^{2k+2}}\bl y=U\bl V.
\end{eqnarray*}
Note that, as always with superflows, the many symmetries amount to monodromy of radicals, and so we tacitly assumed $\sqrt[2k+1]{c^{(2k+1)(2k+2)}}=c^{2k+2}$, since this is compatible with the boundary condition (\ref{init}). Indeed, in our case this tells that
\begin{eqnarray*}
\lim\limits_{t\rightarrow 0}\frac{U(xt,yt)}{t}=
\lim\limits_{t\rightarrow 0}\frac{\sqrt[2k+1]{(tx)^{2k+1}+(ty)^{2k+2}}}{t}=
\lim\limits_{t\rightarrow 0}\sqrt[2k+1]{x^{2k+1}+ty^{2k+2}}=\sqrt[2k+1]{x^{2k+1}}=x.
\end{eqnarray*} 
Thus, to be precise, for $t$ small enough ($|t|<\frac{|x|^{2k+1}}{|y|^{2k+2}}$), we take the unique branch of $\frac{U(xt,yt)}{t}=\sqrt[2k+1]{y^{2k+2}t+x^{2k+1}}$ which is equal to $x$ at $t=0$, and verify the symmetries for such $t$.

\subsection{An exceptional sub-case}Let now $k=0$. Thus, we have the rational flow of level 1 (see \cite{alkauskas} for details)
\begin{eqnarray}
\phi(\m{x})=y^2+x\bl y.\label{amaz}
\end{eqnarray}
As we have seen, this superflow has a $6-$fold cyclic symmetry, and this is generated by the matrix
\begin{eqnarray*}
\alpha=\begin{pmatrix}
\zeta & 0\\
0 & -\zeta^{-1}
\end{pmatrix},\quad \zeta=e^{\frac{2\pi i }{3}}.
\end{eqnarray*}
In general, let
\begin{eqnarray*}
L=\begin{pmatrix}
a & b\\
c & d
\end{pmatrix},\quad ad-bc\neq 1.
\end{eqnarray*}
Then
\begin{eqnarray}
L^{-1}\circ\phi\circ L(x,y)=\frac{d}{ad-bc}(cx+dy)^2+x\bl -\frac{c}{ad-bc}(cx+dy)^2+y.
\label{ll}
\end{eqnarray}
If the flow $\phi$ is invariant under conjugation with $L$, then 
\begin{eqnarray*}
c=0,\quad\frac{d^2}{a}=1.
\end{eqnarray*}
So, all symmetries of this superflow are given by
\begin{eqnarray*}
\widetilde{\Gamma}=\Bigg{\{}\delta_{b,d}=\begin{pmatrix} d^2 & b\\ 0& d\end{pmatrix}:b\in\mathbb{C},d\in\mathbb{C}^{*}\Bigg{\}}.
\end{eqnarray*}
To make the flow look more symmetric, note that $\phi$ is linearly conjugate to the flow $\phi_{\mathrm{sph},\infty}=(x-y)^2+x\bl (x-y)^2+y$. Indeed, in (\ref{ll}) we take $L=\begin{pmatrix}
1 & 0\\
-1 & 1
\end{pmatrix}$. The notation $\phi_{\mathrm{sph},\infty}$ is borrowed from \cite{alkauskas}. Performing this linear conjugation $L^{-1}\circ\phi\circ L$, $L^{-1}\circ\widetilde{\Gamma}\circ L$, we get all but the last statements of the following. 
\begin{prop}Let $\zeta=e^{\frac{2\pi i}{3}}$. The flow
\begin{eqnarray*}
\phi_{\mathrm{sph},\infty}=(x-y)^2+x\bl (x-y)^2+y
\end{eqnarray*}
 has a $6$-fold cyclic symmetry generated by the order $6$ matrix
\begin{eqnarray*}
\gamma=\begin{pmatrix}
\zeta & 0\\
\zeta+\zeta^{-1}&-\zeta^{-1}
\end{pmatrix}.
\end{eqnarray*}
Its vector field is $(x-y)^2\bl (x-y)^2$. This is the unique vector field without denominators with this $6-$fold symmetry. So, $\phi_{\mathrm{sph},\infty}$ is the superflow.\\
 
However, the full group of symmetries of this superflow is the group
\begin{eqnarray*}
\Gamma=\Bigg{\{}\gamma_{d,b}=\begin{pmatrix}
d^2-b & b\\
d^2-d-b & d+b
\end{pmatrix}:b\in\mathbb{C},d\in\mathbb{C}^{*}\Bigg{\}}.
\end{eqnarray*}
Matrix $\gamma_{d,b}$ is of finite order only if $d$ is a root of unity, $b$ is arbitrary, except the case $d=1$, $b\neq 0$, when it is of infinite order. All finite subgroups of $\Gamma$ are cyclic. So, formally, $\phi_{\mathrm{sph}, \infty}$ is still a reducible superflow.  
\end{prop}
\begin{proof} It is more convenient to work with the superflow $y^2+x\bl y$ and the group $\widetilde{\Gamma}$. Eigenvalues of $\delta_{b,d}$ are $d^2$ and $d$, and if they are different and $d$ is a root of unity, a matrix $\delta_{b,d}$ has a diagonal Jordan form, it is of a finite order, and the claim about $\gamma_{d,b}$ follows. \\

Consider the group $\Delta$ generated by the two matrices $\delta_{d,b}$ and $\delta_{e,c}$, where $d,e\neq 1$ are the roots of unity, primitive of orders $p$ and $q$, respectively. Without loss of generality, we may assume that imaginary parts of $d$ and $e$ are positive; otherwise consider $\delta^{-1}_{d,b}$ instead of $\delta_{d,b}$, the same for $\delta_{e,c}$. Suppose, $\Delta$ is a finite group. The root of unity $d^{a}e^{b}$, $a,b\in\mathbb{Z}$, is of order at most $pq$, so its imaginary part is $0$, or is bounded away from $0$. Let $a_{0}$ and $b_{0}$ are chosen such that $d_{0}=d^{a_{0}}e^{b_{0}}$ has a non-negative real part and a positive imaginary part, smallest among all $d^{a}e^{b}$, $a,b\in\mathbb{Z}$.
This is always possible to achieve, unless $d=e=-1$, or $d$ and $e$ are cubic roots of unity; these two cases can be easily treated separately.
Let
\begin{eqnarray*}
\delta_{d_{0},b_{0}}=\delta_{d,b}^{a_{0}}\cdot\delta_{e,c}^{b_{0}}.\quad
\text{Then }\delta_{d_{0},b_{0}}\in\Delta.
\end{eqnarray*}
 We now see that $d$ and $e$ are integral powers of $d_{0}$. Indeed, if for a certain $s\in\mathbb{Z}$, $d_{0}^{s}$, $d$ and $d_{0}^{s+1}$ are three different consecutive points on the unit circle, then the real part of $dd_{0}^{-s}=d^{1-a_{0}s}e^{-b_{0}s}$ is non-negative, and an imaginary part is positive and smaller than that of $d_{0}$, contradicting the minimality condition.  
Thus, $d=d_{0}^{s}$ for a certain $s\in\mathbb{Z}$, and
\begin{eqnarray*}
\delta_{d,b}\delta_{d_{0},b_{0}}^{-s}=\delta_{1,B}\in\Delta,
\end{eqnarray*}
where $B\in\mathbb{C}$. Since this is a matrix of a finite order, this necessarily implies $B=0$. Therefore, $\delta_{d,b}=\delta_{d_{0},b_{0}}^{s}$. Similarly  for $d_{e,c}$. Thus, $\Delta=\{\delta_{d_{0},b_{0}}^{s}:s\in\mathbb{Z}\}$ is a finite cyclic group.
Thus, the group $\Gamma$ does not have non-abelian finite subgroups.
\end{proof}

\subsection{The second special case}
\label{spec-second}
Equally, we can investigate the cyclic group generated by $\alpha$ in case $m=4k+1$. This leads to
\begin{prop}The only vector field with a denominator of degree $\leq 2k-1$ which is  invariant under conjugation with $\alpha$ is given by, up to multiplication with a constant, by
\begin{eqnarray*}
\varpi\bl\varrho=0\bl \frac{x^{2k+1}}{y^{2k-1}}.
\end{eqnarray*}
So this vector field produces the superflow. 
\end{prop}
\begin{proof}We continue in the same vein, and note that the for the denominator $x^{\ell}y^{2k-1-\ell}$ and the term $x^{i}y^{2k+1-i}$ the sums (\ref{sum-pirm}) and (\ref{sum-antr}) remain exactly the same, only the conditions on $i,\ell$ now read as $0\leq i\leq 2k+1$, $0\leq\ell\leq 2k-1$. The answer follows. 
\end{proof}
\begin{prop}
\label{prop5}
Let $k\in\mathbb{N}$. Let us define
\begin{eqnarray*}
\phi(\m{x})=U(x,y)\bl V(x,y)=x\bl\sqrt[2k]{y^{2k}+x^{2k+1}}.
\end{eqnarray*}  
Then $\phi$ has a vector field $0\bl\frac{1}{2k}\frac{x^{2k+1}}{y^{2k-1}}$, and is a projective superflow for the cyclic group of order $8k+2$ generated by the matrix $\alpha$. The full group of symmetries of $\phi$ is given by
\begin{eqnarray*}
\Gamma_{4k+1}=\Bigg{\{}\gamma_{c}=\begin{pmatrix}
c^{2k} & 0\\
0 & c^{2k+1}
\end{pmatrix},\quad c\in\mathbb{C}^{*}\Bigg{\}}.
\end{eqnarray*}
\end{prop} 
As before, we should be careful with ramification, and the symmetry works minding the condition (\ref{init}), and if we assume that for small $t$, $\frac{V(xt,yt)}{t}=\sqrt[2k]{y^{2k}+tx^{2k+1}}$, and the brach is taken which is equal to $y$ for $t=0$.
\subsection{The remaining cases}
Now, if $m=4k$, then $\alpha^{2k}=-I$, and therefore there does not exist a non-trivial vector field with this symmetry.\\

Finally, let $m=4k+2$, $k\in\mathbb{N}$. We can see that this gives a superflow, but not a new one. Indeed, let $\alpha$ be given by (\ref{alfa}), and $\tau=\begin{pmatrix} 0 & 1\\1 & 0\end{pmatrix}$. Then $\tau^{-1}\alpha\tau=\begin{pmatrix}
-\zeta^{-1} & 0\\
0 & \zeta
\end{pmatrix}$. Now, $\xi=-\zeta^{-1}$ is a primitive $(2k+1)$th root of unity, and $-\xi^{-1}=\zeta$. We arrive to the two situations described in Subsections \ref{spec-first} ($m=8k_{0}+6$) and \ref{spec-second} ($m=8k_{0}+2$).
\begin{cor}Let $m\geq 3$, and $\alpha$ be given by (\ref{alfa}). There exists a superflow for the cyclic group generated by $\alpha$ if and only if $m\neq 0 \text{ (mod }4)$.
\end{cor}  
\section{Reducible superflows. II. General case.}
\section{Irreducible superflows. I}
\section{Irreducible superflows. II}    

\end{document}